\documentclass[12pt]{amsart}

\input xy
\xyoption{all}

\usepackage{geometry}
\usepackage{amsmath}
\usepackage{amssymb}
\usepackage{cite}
\usepackage{amsthm}  
\usepackage{mathtools}
\usepackage{amsmath, amsthm, amssymb}
\usepackage{listings}
\usepackage{url}
\usepackage{hyperref}
\usepackage[utf8]{inputenc}
\usepackage{graphicx}

\theoremstyle{definition}
\newtheorem{mydef}{Definition}[section]
\newtheorem{lem}[mydef]{Lemma}
\newtheorem{thm}[mydef]{Theorem}
\newtheorem{cor}[mydef]{Corollary}

\newtheorem{question}[mydef]{Question}
\newtheorem{hypothesis}[mydef]{Hypothesis}
\newtheorem{prop}[mydef]{Proposition}

\newtheorem{example}[mydef]{Example}
\newtheorem{remark}[mydef]{Remark}

\newtheorem{fact}[mydef]{Fact}

\newcommand{\fct}[2]{{}^{#1}#2}

\newcommand{\bigN}{\widehat{N}}

\newcommand{\nsc}[1]{\unionstick^{(#1\text{-ns})}}
\newcommand{\nnsc}[1]{\nunionstick^{(#1\text{-ns})}}
\newcommand{\cf}{\text{cf }}
\newcommand{\seq}[1]{\langle #1 \rangle}

\newcommand{\bkappa}{\bar \kappa}

\newbox\noforkbox \newdimen\forklinewidth
\forklinewidth=0.3pt \setbox0\hbox{$\textstyle\smile$}
\setbox1\hbox to \wd0{\hfil\vrule width \forklinewidth depth-2pt
 height 10pt \hfil}
\wd1=0 cm \setbox\noforkbox\hbox{\lower 2pt\box1\lower
2pt\box0\relax}
\def\unionstick{\mathop{\copy\noforkbox}\limits}
\def\nf{\unionstick}

\setbox0\hbox{$\textstyle\smile$}
\setbox1\hbox to \wd0{\hfil{\sl /\/}\hfil} \setbox2\hbox to
\wd0{\hfil\vrule height 10pt depth -2pt width
               \forklinewidth\hfil}
\newbox\doesforkbox
\setbox\doesforkbox\hbox{\lower 0pt\box1 \lower
2pt\box2\lower2pt\box0\relax}
\def\nunionstick{\mathop{\copy\doesforkbox}\limits}
\def\fork{\nunionstick}

\def\1nf{\unionstick^{(1)}}
\newcommand{\onf}[1]{\unionstick^{(1),#1}}

\def\2nf{\unionstick^{(2)}}
\newcommand\tnf[1]{\unionstick^{(2),#1}}
\def\3nf{\unionstick^{(3)}}

\def\nfm{\unionstick^{-}}
\def\nnfm{\nunionstick^{-}}

\def\ns{\unionstick^{(\text{ns})}}

\def\nes{\unionstick^{(\text{nes})}}

\def\ch{\unionstick^{(\text{ch})}}

\def\chm{\unionstick^{(\bar{\text{ch}})}}

\def\nesm{\unionstick^{(\overline{\text{nes}})}}

\def\forkindep{\mathrel{\raise0.2ex\hbox{\ooalign{\hidewidth$\vert$\hidewidth\cr\raise-0.9ex\hbox{$\smile$}}}}}

\title{Canonical forking in AECs}

\date{\today\\
AMS 2010 Subject Classification: Primary:  03C48, 03C45 and 03C52. Secondary: 03C55,  03C75,  03C85 and 03E55.} 
\keywords{Abstract Elementary Classes; forking; Classification Theory; stability; good frames}

\author{Will Boney}
\email{wboney@math.harvard.edu}
\urladdr{http://math.harvard.edu/\textasciitilde wboney/}
\address{Mathematics Department, Harvard University, Cambridge, MA, USA}
\thanks{This material is based upon work done while the first author was supported by the National Science Foundation under Grant No. DMS-1402191.}

\author{Rami Grossberg}
\email{rami@cmu.edu}
\address{Department of Mathematical Sciences, Carnegie Mellon University, Pittsburgh, PA, USA}

\author{Alexei Kolesnikov}
\email{akolesnikov@towson.edu}
\address{Department of Mathematics, Towson University, MD, USA}

\author{Sebastien Vasey}
\thanks{This material is based upon work done while the fourth author was supported by the Swiss National Science Foundation under Grant No.\ 155136.}
\email{sebv@cmu.edu}
\urladdr{http://math.cmu.edu/\textasciitilde svasey/}
\address{Department of Mathematical Sciences, Carnegie Mellon University, Pittsburgh, PA, USA}

\begin{document}

\begin{abstract}
Boney and Grossberg \cite{bg-v9} proved that every nice AEC has an independence relation. We prove that this relation is unique: in any given AEC, there can exist at most one independence relation that satisfies existence, extension, uniqueness and local character. While doing this, we study more generally the properties of independence relations for AECs and also prove a canonicity result for Shelah's good frames. The usual tools of first-order logic (like the finite equivalence relation theorem or the type amalgamation theorem in simple theories) are not available in this context. In addition to the loss of the compactness theorem, we have the added difficulty of not being able to assume that types are sets of formulas. We work axiomatically and develop new tools to understand this general framework.
\end{abstract}

\maketitle

\tableofcontents

\section{Introduction}


Let $K$ be an abstract elementary class (AEC) which satisfies amalgamation, joint embedding, and which does not have maximal models. These assumptions allow us to work inside its monster model $\mathfrak{C}$. The main results of this paper are:

\begin{enumerate}
  \item\label{abstract-canon} There is at most one independence relation satisfying existence, extension, uniqueness and local character (Corollary \ref{endcor-improved}).
  \item\label{ch-canon} Under some reasonable conditions, the coheir relation of \cite{bg-v9} has local character and is canonical (Theorems \ref{nflocal} and \ref{canon-coheir}).
  \item\label{weak-successful-canon} Shelah's weakly successful good $\lambda$-frames are canonical: an AEC can have at most one such frame (Theorem \ref{canon-good-frames}).
\end{enumerate}

To understand the relevance of the results, some history is necessary.

In 1970, Shelah discovered the notion ``$\text{tp} (\bar{a} / B)$ forks over $A$'' (for $A \subseteq B$), a generalization of Morley's rank in $\omega$-stable theories. Its basic properties were published in \cite{shelahfobook78}. 

In 1974, Lascar \cite[Theorem 4.9]{lascar76} established that for superstable theories, any relation between $\bar{a}$, $B$, $A$ satisfying the basic properties of forking is Shelah's forking relation. In 1984, Harnik and Harrington \cite[Theorem 5.8]{hh84} extended Lascar's abstract characterization to stable theories. Their main device was the finite equivalence relation theorem. In 1997, Kim and Pillay \cite[Theorem 4.2]{kp97} published an extension to simple theories, using the independence theorem (also known as the type-amalgamation theorem).

This paper deals with the characterization of independence relations in various non-elementary classes. An early attempt on this problem can be found in Koles\-ni\-kov's \cite{kolesnikov-dependence}, which focuses on some important particular cases (e.g.\ homogeneous model theory and classes of atomic models). We work in a more general context, and only rely on the abstract properties of independence. We cannot assume that types are sets of formulas, so work only with Galois (i.e.\ orbital) types.

In \cite[Chapter II]{sh300-orig} (which later appeared as \cite[Chapter V.B]{shelahaecbook2}), Shelah gave the first axiomatic definition of independence in AECs, and showed that it generalized first-order forking. In \cite[Chapter II]{shelahaecbook}, Shelah gave a similar definition, localized to models of a particular size $\lambda$ (the so-called ``good $\lambda$-frame''). Shelah proved that a good frame existed, under very strong assumptions (typically, the class is required to be categorical in two consecutive cardinals). 

Recently, working with a different set of assumptions (the existence of a monster model and tameness), Boney and Grossberg \cite{bg-v9} gave conditions (namely a form of Galois stability and the extension property for coheir) under which an AEC has a global independence relation. This showed that one could study independence in a broad family of AECs. Our paper is strongly motivated by both \cite[Chapter II]{shelahaecbook} and \cite{bg-v9}. 

The paper is structured as follows. In Section \ref{notation-prereq}, we fix our notation, and review some of the basic concepts in the theory of AECs. In Section \ref{indep-rel}, we introduce independence relations, the main object of study of this paper, as well as some important properties they could satisfy, such as extension and uniqueness. We consider two examples: coheir and nonsplitting.

In Section \ref{comparing}, we prove a weaker version of (\ref{abstract-canon}) (Corollary \ref{endcor}) that has some extra assumptions. This is the core of the paper.

In Section \ref{relations-properties}, we go back to the properties listed in Section \ref{indep-rel} and investigate relations between them. We show that some of the hypotheses in Corollary \ref{endcor} are redundant. For example, we show that the symmetry and transitivity properties follow from existence, extension, uniqueness, and local character. We conclude by proving (\ref{abstract-canon}). Finally, in Section \ref{last-section}, we apply our methods to the coheir relation considered in \cite{bg-v9} and to Shelah's good frames, proving (\ref{ch-canon}) and (\ref{weak-successful-canon}).

While we work in a more general framework, the basic results of Sections \ref{notation-prereq}-\ref{indep-rel} often have proofs that are very similar to their first-order analogs. Readers feeling confident in their knowledge of first-order nonforking can start reading directly from Section \ref{comparing} and refer back to Sections \ref{notation-prereq}-\ref{indep-rel} as needed.

This paper was written while the first and fourth authors were working on a Ph.D.\ under the direction of Rami Grossberg at Carnegie Mellon University. They would like to thank Professor Grossberg for his guidance and assistance in their research in general and in this work specifically.

An early version of this paper was circulated already in early 2014. Since that time, Theorem \ref{symmetry} has been used by the fourth author to build a good frame from amalgamation, tameness, and categoricity in a suitable cardinal \cite{ss-tame-toappear-v3}. With VanDieren, the fourth author has also used it to deduce a certain symmetry property for nonsplitting in classes with amalgamation categorical in a high-enough cardinal \cite{vv-symmetry-transfer-v3}, with consequences on the uniqueness of limit models. The question of canonicity of forking in more local setups (e.g.\ when the independence relation is only defined for certain types over models of a certain size) is pursued further in \cite{indep-aec-v5}. The latter preprint addresses Questions \ref{min-closure-q}, \ref{categ-good-frame-q}, \ref{indep-existence-q}, and \ref{coheir-canon-q} posed in this paper.

\section{Notation and prerequisites}\label{notation-prereq}

We assume the reader is familiar with abstract elementary classes and the basic related concepts. We briefly review what we need in this paper, and set up some notation.

\begin{hypothesis}\label{monster-model-hyp}
We work in a fixed abstract elementary class $(K, \prec)$ which satisfies amalgamation and joint embedding, and has no maximal models. 
\end{hypothesis}

\subsection{The monster model}

\begin{mydef}
Let $\mu > \text{LS} (K)$ be a cardinal. For models $M \prec N$, we say $N$ is a \emph{$\mu$-universal extension of $M$} if for any $M' \succ M$, with $\| M'\|  < \mu$, $M'$ can be embedded inside $N$ over $M$, i.e. there exists a $K$-embedding $f: M' \rightarrow N$ fixing $M$ pointwise. We say $N$ is a \emph{universal extension of $M$} if it is a $\| M\| ^+$-universal extension of $M$.
\end{mydef}
\begin{mydef}
Let $\mu > \text{LS} (K)$ be a cardinal. We say a model $N$ is \emph{$\mu$-model homogeneous} if for any $M \prec N$, $N$ is a $\mu$-universal extension of $M$. We say $M$ is \emph{$\mu$-saturated} if it is $\mu$-model homogeneous (this is equivalent to the classical definition by \cite[Lemma 0.26]{sh576}).
\end{mydef}

\begin{mydef}[Monster model]
Since $K$ has amalgamation and joint embedding properties and has no maximal models, we can build a strictly increasing continuous chain $(\mathfrak{C}_i)_{i \in \text{OR}}$, where for all $i$, $\mathfrak{C}_{i + 1}$ is universal over $\mathfrak{C}_i$. We call the union $\mathfrak{C} := \bigcup_{i \in \text{OR}} \mathfrak{C}_i$ the \emph{monster model\footnote{Since $\mathfrak{C}$ is a proper class, it is strictly speaking not an element of $K$. We ignore this detail, since we could always replace $\text{OR}$ in the definition of $\mathfrak{C}$ by a cardinal much bigger than the size of the models under discussion.} of $K$}. 

Any model of $K$ can be embedded inside the monster model, so we will adopt the convention that \emph{any set or model we consider is a subset or a substructure of $\mathfrak{C}$}. 

We write $\text{Aut}_A (\mathfrak{C})$ for the set of automorphisms of $\mathfrak{C}$ fixing $A$ pointwise. When $A = \emptyset$, we omit it.
\end{mydef}

We will use the following without comments.

\begin{remark}
  Let $M$, $N$ be models. By our convention, $M \prec \mathfrak{C}$ and $N \prec \mathfrak{C}$, thus by the coherence axiom, $M \subseteq N$ implies $M \prec N$.
\end{remark}

\begin{mydef}
Let $I$ be an index set. Let $\bar{A} := (A_i)_{i \in I}$, $\bar{B} := (B_i)_{i \in I}$ be sequences of sets, and let $C$ be a set. We write $f: \bar{A} \equiv_C \bar{B}$ to mean that $f \in \text{Aut}_C (\mathfrak{C})$, and for all $i \in I$, $f[A_i] = B_i$. We write $\bar{A} \equiv_C \bar{B}$ to mean that $f: \bar{A} \equiv_C \bar{B}$ for some $f$. When $C$ is empty, we omit it. 

We will most often use this notation when $I$ has a single element, or when all the sets are singletons. In the later case, we identify a set with the corresponding singleton, i.e.\ if $\bar{a} = (a_i)_{i \in I}$ and $\bar{b} := (b_i)_{i \in I}$ are sequences, we write $f: \bar{a} \equiv_C \bar{b}$ instead of $f: \bar{A} \equiv_C \bar{B}$, with $A_i := \{a_i\}$, $B_i := \{b_i\}$. We write $\text{gtp} (\bar{a} / C)$ for the $\equiv_C$ equivalence class of $\bar{a}$. This corresponds to the usual notion of Galois types first defined in \cite[Definition 0.17]{sh576}.

Note that for sets $A, B$, we have $f: A \equiv_C B$ precisely when there are enumerations $\bar{a}$, $\bar{b}$ of $A$ and $B$ respectively such that $f: \bar{a} \equiv_C \bar{b}$.
\end{mydef}

\subsection{Tameness and stability}

Although we will make no serious use of it in this paper, we briefly review the notion of tameness. While it appears implicitly in \cite{sh394}, tameness was first introduced in \cite{tamenessone} and used in \cite{tamenessthree} to prove an upward categoricity transfer. Our definition follows \cite[Definition 3.1]{tamelc-jsl}.

\begin{mydef}[Tameness]
  Let $\kappa > \text{LS} (K)$. Let $\alpha$ be a cardinal. We say $K$ is \emph{$\kappa$-tame for $\alpha$-length types} if for any tuples $\bar{a}, \bar{b}$ of length $\alpha$, and any $M \in K$, if $\bar{a} \not \equiv_M \bar{b}$, there exists $M_0 \prec M$ of size $\le \kappa$ such that $\bar{a} \not \equiv_{M_0} \bar{b}$. 

  We say $K$ is \emph{$(<\kappa)$-tame for $\alpha$-length types} if for any tuples $\bar{a}, \bar{b}$ of length $\alpha$, and any $M \in K$, if $\bar{a} \not \equiv_M \bar{b}$, there exists $M_0 \prec M$ of size $< \kappa$ such that $\bar{a} \not \equiv_{M_0} \bar{b}$. 

  We say $K$ is \emph{$\kappa$-tame} if it is $\kappa$-tame for $1$-length types. We say $K$ is \emph{fully $\kappa$-tame} if it is $\kappa$-tame for all lengths. Similarly for $(<\kappa)$-tame.
\end{mydef}

The following dual of tameness is introduced in \cite[Definition 3.3]{tamelc-jsl}:

\begin{mydef}[Type shortness]
  Let $\kappa > \text{LS} (K)$. Let $\mu$ be a cardinal. We say $K$ is \emph{$\kappa$-type short over $\mu$-sized models} if for any index set $I$, any enumerations $\bar{a} := (a_i)_{i \in I}$, $\bar{b} := (b_i)_{i \in I}$ of type $I$, and any $M \in K_\mu$, if $\bar{a} \not \equiv_M \bar{b}$, there is $I_0 \subseteq I$ of size $\le \kappa$ such that $\bar{a}_{I_0} \not \equiv_M \bar{b}_{I_0}$. Here $\bar{a}_{I_0} := (a_i)_{i \in I_0}$. 

  We define $(<\kappa)$-type short over $\mu$-sized models similarly.

  We say $K$ is \emph{fully $\kappa$-type short} if it is $\kappa$-type short over $\mu$-sized models for all $\mu$. Similarly for $(<\kappa)$-type short.
\end{mydef}

We also recall that we can define a notion of stability:

\begin{mydef}[Stability]
  Let $\lambda \ge \text{LS} (K)$ and $\alpha$ be cardinals. We say $K$ is \emph{$\alpha$-stable in $\lambda$} if for any $M \in K_\lambda$, $S^\alpha (M) := \{\text{gtp} (\bar{b} / M) \mid \bar{b} \in \fct{\alpha}{\mathfrak{C}}\}$ has cardinality $\le \lambda$. Equivalently, given any collection $\{A_i\}_{i < \lambda^+}$, where for all $i < \lambda^+$, $|A_i| = \alpha$, there exists $i < j$ such that $A_i \equiv_M A_j$.

  We say $K$ is \emph{stable} in $\lambda$ if it is $1$-stable in $\lambda$.

  We say $K$ is \emph{$\alpha$-stable} if it is $\alpha$-stable in $\lambda$ for some $\lambda \ge \text{LS} (K)$. We say $K$ is \emph{stable} if it is $1$-stable in $\lambda$ for some $\lambda \ge \text{LS} (K)$. We write ``unstable'' instead of ``not stable''.
\end{mydef}

\begin{remark}\label{monot-stab}
  If $\alpha < \beta$, and $K$ is $\beta$-stable in $\lambda$, then $K$ is $\alpha$-stable in $\lambda$.
\end{remark}

The following follows from \cite[Theorem 3.1]{longtypes-toappear-v2}.

\begin{fact}\label{stab-longtypes}
  Let $\lambda \ge \text{LS} (K)$. Let $\alpha$ be a cardinal. Assume $K$ is stable in $\lambda$ and $\lambda^\alpha = \lambda$. Then $K$ is $\alpha$-stable in $\lambda$.
\end{fact}

\section{Independence relations}\label{indep-rel}

In this section, we define independence relations, the main object of study of this paper. We then consider two examples: coheir and nonsplitting.

\subsection{Basic definitions}

\begin{mydef}[Independence relation]\label{indep-rel-def}
  An \emph{independence relation} $\nf$ is a set of triples of the form $(A, M, N)$ where $A$ is a set, $M, N$ are models (i.e.\ $M, N \in K$), $M \prec N$. Write $A \nf_M N$ for $(A, M, N) \in \nf$. When $A = \{a\}$, we may write $a \nf_M N$ for $A \nf_M N$. We require that $\nf$ satisfies the following properties:

  \begin{itemize}
    \item (I) Invariance: Assume $(A, M, N) \equiv (A', M', N')$. Then $A \nf_M N$ if and only if $A' \nf_{M'} N'$.
    \item (M) Left and right monotonicity: If $A \nf_M N$, $A' \subseteq A$, $M \prec N' \prec N$, then $A' \nf_M N'$.
    \item (B) Base monotonicity: If $A \nf_M N$, and $M \prec M' \prec N$, then $A \nf_{M'} N$. 
  \end{itemize}

  We write $\nf_M$ for $\nf$ restricted to the base set $M$, and similarly for e.g. $A \nf_M$.
\end{mydef}

In what follows, $\nf$ always denotes an independence relation. 

\begin{remark}
To avoid relying on a monster model, we could introduce an ambient model $\bigN$ as a fourth parameter in the above definition (i.e.\ we would write $A \nf_M^{\bigN} N$). This would match the approach in \cite[Chapter V.B]{shelahaecbook2} and \cite[Chapter II]{shelahaecbook} where the existence of a monster model is not assumed. We would require that $\bigN$ contains the other parameters $A$, $M$ and $N$. To avoid cluttering the notation, we will not adopt this approach, but generalizing most of our results to this context should cause no major difficulty. Some simple cases will be treated in the discussion of good frames in Section \ref{last-section}. In \cite{indep-aec-v5}, many of the results of this paper are stated in a ``monsterless'' framework.
\end{remark}

We will consider the following properties of independence\footnote{Continuity, transitivity, uniqueness, existence and extension are adapted from \cite{makkaishelah}. Symmetry comes from \cite[Chapter II]{shelahaecbook}.}:

\begin{itemize}
  \item $(C)_{\kappa}$ Continuity: If $A \fork_M N$, then there exists $A^- \subseteq A$, $B^- \subseteq N$ of size strictly less than $\kappa$ such that for all $N_0 \succ M$ containing $B^-$, $A^- \fork_M N_0$.
  \item $(T)$ Left transitivity: If $M_1 \nf_{M_0} N$, and $M_2 \nf_{M_1} N$, with $M_0 \prec M_1 \prec M_2$, then $M_2 \nf_{M_0} N$.
  \item $(T_\ast)$ Right transitivity: If $A \nf_{M_0} M_1$, and $A \nf_{M_1} M_2$, with $M_0 \prec M_1 \prec M_2$, then $A \nf_{M_0} M_2$.
  \item $(S)$ Symmetry: If $A \nf_M N$, then there is $M' \succ M$ with $A \subseteq M'$ such that $N \nf_M M'$. If $A$ is a model extending $M$, one can take $M' = A$ \footnote{This second part actually follows from monotonicity and the first part.}.
  \item $(U)$ Uniqueness: If $A \nf_M N$, $A' \nf_M N$, and $f: A \equiv_M A'$, then  $g: A \equiv_N A'$ for some $g$ so that $g \upharpoonright A = f \upharpoonright A$.
  \item $(E)$ The following properties hold:
    \begin{itemize}
      \item $(E_0)$ Existence: for all sets $A$ and models $M$, $A \nf_M M$.
      \item $(E_1)$ Extension: Given a set $A$, and $M \prec N \prec N'$, if $A \nf_M N$, then there is $A' \equiv_N A$ such that $A' \nf_M N'$.
    \end{itemize}
  \item $(L)$ Local character: $\kappa_\alpha (\nf) < \infty$ for all $\alpha$, where $\kappa_\alpha(\nf) := \min \{ \lambda \in \operatorname{REG} \cup \{\infty\} :$ for all $\mu = \cf \mu \geq \lambda$, all increasing, continuous chains $ \seq{M_i : i \le \mu}$ and all sets $A$ of size $\alpha$, there is some $ i_0 < \mu $ so $ A \nf_{M_{i_0}} M_\mu\}$.
  \item $(E_+)$ Strong extension: A technical property used in the proof of canonicity. See Definition \ref{eplus-def}. 
\end{itemize}

For $(P)$ a property that is not local character, and $M$ a model, when we say $\nf$ has $(P)_M$, we mean $\nf_M$ has $(P)$ (i.e. $\nf$ has $(P)$ when the base is restricted to be $M$). If $P$ is either $(T)$ or $(T_\ast)$, $(P)_M$ means we assume $M_0 = M$ in the definition.

Whenever we are considering two independence
relations $\1nf$ and $\2nf$, we write $(P^{(1)})$ as a shorthand
for ``$\1nf$ has $(P)$'', and similarly for $(P^{(2)})$.

Notice the following important consequence of $(E)$:

\begin{remark}\label{ext-conseq}
  Assume $\nf$ has $(E)_M$. Then  for any $A$, and $N \succ M$, there is $A' \equiv_M A$ such that $A' \nf_M N$ (use $(E_0)_M$ to see $A \nf_M M$, and then use $(E_1)_M$). 

  Assuming $(T^\ast)_M$, this last statement is actually equivalent to $(E)_M$.  
\end{remark}

The property $(E_+)$ will be introduced and motivated later in the paper. For now, we note that there is an asymmetry in our definition of an independence relation: the parameter on the left is allowed to be an arbitrary set, while the parameter on the right must be a model extending the base. This is because we have in mind the analogy ``$a \nf_M N$ if and only if $\text{tp} (a / N)$ does not fork over $M$'', and in AECs, types over models are much better behaved than types over sets. 

The price to pay is that the statement of symmetry is not easy to work with. Assume for example we know an independence relation satisfies $(T)$ and $(S)$. Should it satisfy $(T_\ast)$? Surprisingly, this is not easy to show. We prove it in Lemma \ref{right-trans-sym}, assuming $(E)$. For now, we prepare the ground by showing how to extend an independence relation to take arbitrary sets on the right hand side.

\begin{mydef}[Closure of an independence relation]\label{closuredef}
  We call $\nfm$ a \emph{closure} of $\nf$ if $\nfm$ is a relation defined on all triples of the form $(A, M, B)$, where $M$ is a model (but maybe $M \not \subseteq B$). We require it satisfies the following properties:

\begin{itemize}
  \item For all $A$, and all $M \prec N$, $A \nf_M N$ if and only if $A \nfm_M N$.
  \item $(I)$ Invariance: If $(A, M, B) \equiv (A', M', B')$, then $A \nfm_M B$ if and only if $A' \nfm_{M'} B'$.
  \item $(M)$ Left and right monotonicity: If $A \nfm_M B$ and $A' \subseteq A$, $B' \subseteq B$, then $A' \nfm_M B'$.
  \item $(B)$ Base monotonicity: If $A \nfm_M B$, and $M \prec M' \subseteq M \cup B$, then $A \nfm_{M'} B$.
\end{itemize}

The \emph{minimal closure} of $\nf$ is the relation $\overline{\nf}$ defined by $A \overline{\nf_M} C$ if and only if there exists $N \succ M$, with $C \subseteq N$, so that $A \nf_M N$.
\end{mydef}

It is straightforward to check that the minimal closure of $\nf$ is the smallest closure of $\nf$ but there might be others (and they also sometimes turn out to be useful), see the coheir and explicit nonsplitting examples below.

We can adapt the list of properties to a closure $\nfm$.

\begin{mydef} \
  \begin{itemize}
    \item We say $\nfm$ has $(S)$ if for all sets $A, B$, $A \nfm_M B$ if and only if $B \nfm_M A$. 
    \item We say that $\nfm$ has $(C)_\kappa$ if whenever $A \nnfm_M B$, there exists $A^- \subseteq A$, $B^- \subseteq B$ of size strictly less than $\kappa$ such that $A^- \nnfm_M B^-$.
    \item We say that $\nfm$ has $(E_1)$ if whenever $A \nfm_M C$, and $C \subseteq C'$, there exists $A' \equiv_{MC} A$ such that $A' \nfm_M C'$.
    \item We say that $\nfm$ has $(U)$ if whenever $A \nfm_M C$, $A' \nfm_M C$, and $f: A \equiv_M A'$, there is $g: A \equiv_{MC} A'$ with $g \upharpoonright A = f \upharpoonright A$.
    \item We say that $\nfm$ has $(T)$ if whenever $M_0 \prec M_1 \prec M_2$, $M_2 \nfm_{M_1} C$, and $M_1 \nfm_{M_0} C$, we have $M_2 \nfm_{M_0} C$.
    \item The statements of $(T_\ast)$, $(E_0)$, $(L)$ are unchanged. We will not need to use $(E_+)$ on a closure.
  \end{itemize}
\end{mydef}

For an arbitrary closure, we cannot say much about the relationship between the properties satisfied by $\nf$ and those satisfied by $\nfm$. The situation is different for the minimal closure, but we defer our analysis to section \ref{relations-properties}.

\begin{remark}
Shelah's notion of a good $\lambda$-frame introduced in \cite[Chapter II]{shelahaecbook} is another axiomatic approach to independence in AECs. There are several key differences with our framework.  In particular, good $\lambda$-frames only operate on $\lambda$-sized models and singleton sets.  On the other hand, the theory of good $\lambda$-frames is very developed; see e.g.\ \cite{shelahaecbook,jasi,jrsh875}.

An earlier framework which is closer to our own is the ``Existential framework'' $\text{AxFr}_3$ (see \cite[Definition V.B.1.9]{shelahaecbook2}). The key differences are that $\text{AxFr}_3$ only defines $M_1 \nf_M^{\bigN} M_2$ when $M \prec M_\ell$, $\ell = 1, 2$, $\text{AxFr}_3$ (essentially) assumes $(C)_{\aleph_0}$, while we seldom need continuity, and local character (a property crucial to our canonicity proof) is absent from the axioms of $\text{AxFr}_3$.
\end{remark}

\subsection{Examples}

Though so far developed abstractly, this framework includes many previously studied independence relations. 

\begin{mydef}[Coheir, \cite{bg-v9}]
  Fix a cardinal $\kappa > \text{LS} (K)$. We call a set small if it is of size less than $\kappa$.  For $M \prec N$, define
  
  \begin{align*}
    A \ch_M N \iff &\text{ for every small $A^- \subseteq A$ and $N^- \prec N$, } \\
    &\text{ there is $B^- \subseteq M$ such that $B^- \equiv_{N^-} A^-$.}
  \end{align*}
\end{mydef}

One can readily check that $\ch$ satisfies the properties of an independence relation. $\ch$ was first studied in \cite{bg-v9}, based on results of \cite{makkaishelah} and  \cite{tamelc-jsl}, and generalizes the first-order notion of coheir.  An alternative name for this notion is \emph{$(<\kappa)$ satisfiability}.  Sufficient conditions for this relation to be well-behaved (i.e.\ to have most of the properties listed above) are given in \cite[Theorem 5.1]{bg-v9}, reproduced here as Fact \ref{bg-facts}.

\begin{mydef}
We define a natural closure for $\ch$:

  \begin{align*}
    A \chm_M C \iff &\text{ for every small $A^- \subseteq A$ and $C^- \subseteq C$, } \\
    &\text{ there is $B^- \subseteq M$ such that $B^- \equiv_{C^-} A^-$.}
  \end{align*}
\end{mydef}

It is straightforward to check that $\chm$ is indeed a closure of $\ch$, but it is not clear at all that this is the \emph{minimal} one. This closure will be useful in the proof of local character (Theorem \ref{nflocal}) Note that $\chm$ differs from the notion of coheir given in \cite{makkaishelah}; there, types are consistent sets of formulas from a fragment of $L_{\kappa, \kappa}$ for $\kappa$ strongly compact and the notion there (see \cite[Definition 4.5]{makkaishelah}) allows parameters from $C$ and $|M|$. 

\begin{mydef}[$\mu$-nonsplitting, \cite{sh394}]
Let $\mu \geq LS(K)$.  For $M \prec N$, we say $A \nsc{\mu}_M N$ if and only if for for all $N_1, N_2 \in K_{\le \mu}$ with $M \prec N_\ell \prec N$, $\ell = 1, 2$, if $f: N_1 \equiv_M N_2$, then there is $g: N_1 \equiv_{AM} N_2$ such that $f \upharpoonright N_1 = g \upharpoonright N_1$.
\end{mydef}

There is also a definition of nonsplitting that does not depend on a cardinal $\mu$.

\begin{mydef}[Nonsplitting]
For $M \prec N$, $$A \ns_M N \iff \text{ $A \nsc{\mu}_M N$ for all $\mu$.}$$
\end{mydef}

An equivalent definition of nonsplitting is given by the following.

\begin{prop} \label{ns-equiv-def} \
  $A \ns_M N$ if and only if for all $N_1, N_2 \in K$ with $M \prec N_\ell \prec N$, $\ell = 1,2$, if $h: N_1 \equiv_M N_2$, then $f: A \equiv_{N_2} h[A]$ for some $f$ with $f \upharpoonright A = h \upharpoonright A$ (equivalently, $\bar{a} \equiv_{N_2} h (\bar{a})$ for all enumerations $\bar{a}$ of $A$).

The analog statement also holds for $\mu$-nonsplitting.
\end{prop}
\begin{proof}
  Assume $h: N_1 \equiv_M N_2$, and $f: A \equiv_{N_2} h[A]$ is such that $f \upharpoonright A = h \upharpoonright A$. Let $g := f^{-1} \circ h$. Then $g \upharpoonright N_1 = h \upharpoonright N_1$, and $g$ fixes $AM$. In other words, $g: N_1 \equiv_{A M} N_2$ is as needed. Conversely, assume $h: N_1 \equiv_M N_2$. Find $g: N_1 \equiv_{AM} N_2$ such that $h \upharpoonright N_1 = g \upharpoonright N_1$. Then $f := h \circ g^{-1}$ is the desired witness that $A  \equiv_{N_2} h[A]$.
\end{proof}

Using Proposition \ref{ns-equiv-def} to check base monotonicity, it is easy to see that both $\ns$ and $\nsc{\mu}$ are independence relations. These notions of splitting in AECs were first explored in \cite{sh394}, but have seen a wide array of uses; see \cite{shvi635,vandierennomax,nomaxerrata,gvv-toappear-v1_2,ss-tame-toappear-v3} for examples.  $\mu$-nonsplitting is more common in the literature, but we focus on nonsplitting here. Using tameness, there is a correspondence between the two:

\begin{prop}\label{ns-tameness}
Let $M \prec N$ and $\mu \geq LS(K)$. If $K$ is $\mu$-tame for $|A|$-length types and $\mu' \in [\mu, \|N\|]$, then

$$A \nsc{\mu}_M N \implies A \nsc{\mu'}_M N$$
\end{prop}
\begin{proof} \
We use the equivalence given by Proposition \ref{ns-equiv-def}. Let $\mu' \in [\mu, \| N\| ]$, and suppose $A \nnsc{\mu'}_M N$. Then there are $N_\ell \in K_{\mu'}$ so $M \prec N_\ell \prec N$ for $\ell = 1, 2$ and $h: N_1 \equiv_M N_2$, but $\bar{a} \not \equiv_{N_2} h (\bar{a})$ for some enumeration $\bar{a}$ of $A$. By tameness, there is $N_2^- \in K_{\le \mu}$ so that $\bar{a} \not \equiv_{N_2^-} h(\bar{a})$. Without loss of generality, $M \prec N_2^-$. Let $N_1^- := h^{-1}[N_2^-]$. Then $N_1^-$ and $N_2^-$ witness that $A \nnsc{\mu}_M N$.
\end{proof}

A variant is explicit nonsplitting, which allows the $N_i$'s to be sets instead of requiring models; this is based on explicit non-strong splitting from \cite[Definition 4.11.2]{sh394}.

\begin{mydef}[Explicit Nonsplitting]
For $M \prec N$, we say $A \nes_M N$ if and only if for for all $C_1, C_2 \subseteq N$, if $f: C_1 \equiv_M C_2$, then there is $g: C_1 \equiv_{AM} C_2$ such that $f \upharpoonright C_1 = g \upharpoonright C_1$.
\end{mydef}

From the definition, we see immediately that $\nes \subseteq \ns$. Of course, the corresponding version of Proposition \ref{ns-equiv-def} also holds for $\nes$, so it is again straightforward to check that $\nes$ is an independence relation. One advantage of using $\nes$ over $\ns$ is that it has a natural closure:

\begin{mydef}
  We say $A \nesm_M C$ if and only if for for all $C_1, C_2 \subseteq C$, if $f: C_1 \equiv_M C_2$, then there is $g: C_1 \equiv_{AM} C_2$ such that $f \upharpoonright C_1 = g \upharpoonright C_1$.
\end{mydef}

Again, it is not clear this is the minimal closure. We will have no use for this closure, so for most of the paper we will stick with regular nonsplitting. 

Nonsplitting will be used mostly as a technical tool to state and prove intermediate lemmas, while coheir will be relevant only in Section \ref{last-section}.

\subsection{Properties of coheir and nonsplitting}

We now investigate the properties satisfied by coheir and nonsplitting. Here is what holds in general:

\begin{prop}\label{elem-prop-ch-ns}
  Let $\kappa > \text{LS} (K)$.

  \begin{enumerate}
    \item $\ch$ and $\chm$ have $(C)_\kappa$, and $(T)$.
    \item If $M$ is $\kappa$-saturated, $\ch$ and $\chm$ have $(E_0)_M$.
    \item $\ns$, $\nes$, and $\nesm$ have $(E_0)$.
  \end{enumerate}
\end{prop}
\begin{proof}
  Just check the definitions.
\end{proof}

While extension and uniqueness are usually considered very strong assumptions, it is worth noting that nonsplitting satisfies a weak version of them, see \cite[Theorems I.4.10, I.4.12]{vandierennomax}. It is also well known that nonsplitting has local character assuming tameness and stability (see e.g.\ \cite[Fact 4.6]{tamenessone}). This will not be used.

Regarding coheir, the following\footnote{Since this paper was first submitted, a stronger result has been proven (for example one need not assume $(E)$). See \cite[Theorem 5.15]{sv-infinitary-stability-v6}.} appears in \cite{bg-v9} :

\begin{fact}\label{bg-facts}
  Let $\kappa > \text{LS} (K)$ be regular. Assume $K$ is fully $(<\kappa)$-tame, fully $(<\kappa)$-type short, has no weak $\kappa$-order property\footnote{See \cite[Definition 4.2]{bg-v9}.} and $\ch$ has $(E)$\footnote{All the properties mentioned in this Lemma are valid for models of size $\ge \kappa$ only.}.

  Then $\ch$ has $(U)$ and $(S)$.

  Moreover, if $\kappa$ is strongly compact, then the tameness and type-shortness hypotheses hold for free, $\ch$ has $(E_1)$, and ``no weak $\kappa$ order property'' is implied by ``$\exists \lambda > \kappa$ so $I(\lambda, K) < 2^\lambda$.''
\end{fact}

As we will see, right transitivity $(T_\ast)$ can be deduced either from symmetry and $(T)$ (Lemma \ref{right-trans-sym}) or from uniqueness (Lemma \ref{ext-uniq-trans}). Local character will be shown to follow from symmetry (Theorem \ref{nflocal}). 

\section{Comparing two independence relations}\label{comparing}

In this section, we prove the main result of this paper (canonicity of forking), modulo some extra hypotheses that will be eliminated in Section \ref{relations-properties}. After discussing some preliminary lemmas, we introduce a strengthening  of the extension property, $(E_+)$, which plays a crucial role in the proof. We then prove canonicity using $(E_+)$ (Corollary \ref{eplus-cor}). Finally, we show $(E_+)$ follows from some of the more classical properties that we had previously introduced (Corollary \ref{eplus}), obtaining the main result of this section (Corollary \ref{endcor}). We conclude by giving some examples showing our hypotheses are close to optimal.

For the rest of this section, we fix two independence relations $\1nf$ and $\2nf$. Recall from Definition \ref{indep-rel-def} that this means they satisfy $(I)$, $(M)$ and $(B)$. We aim to show that if $\1nf$ and $\2nf$ satisfy enough of the properties introduced in Section \ref{indep-rel}, then $\1nf = \2nf$.

The first easy observation is that given some uniqueness, only one direction is necessary\footnote{Shelah states as an exercise a variation of this lemma in \cite[Exercise II.6.6.(1)]{shelahaecbook}.}:

\begin{lem}\label{onedir}
  Let $M$ be a model. Assume:

  \begin{enumerate}
    \item $\1nf_M \subseteq \2nf_M$
    \item $(E^{(1)})_M, (U^{(2)})_M$
  \end{enumerate}

  Then $\1nf_M = \2nf_M$.
\end{lem}
\begin{proof}
  Assume $A \2nf_M N$. By $(E^{(1)})_M$, find $A' \equiv_M A$ so that $A' \1nf_M N$. By hypothesis (1), $A' \2nf_M N$. By $(U^{(2)})_M$, $A' \equiv_N A$. By $(I^{(1)})_M$, $A \1nf_M N$.
\end{proof}

With a similar idea, one can relate an arbitrary independence relation to nonsplitting\footnote{Shelah gives a variation of this lemma in \cite[Claim III.2.20.(1)]{shelahaecbook}.}:

\begin{lem}\label{nsmon}
  Assume $(U)_M$. Then $\nf_M \subseteq \ns_M$.
\end{lem}
\begin{proof}
  Assume $A \nf_M N$. Let $M \prec N_1, N_2 \prec N$ and $h: N_1 \equiv_M N_2$. By monotonicity, $A \nf_M N_\ell$ for $\ell = 1, 2$. By invariance, $h [A] \nf_M N_2$. By $(U)_M$, there is $f: A \equiv_{N_2} h[A]$ with $f \upharpoonright A = h \upharpoonright A$. By Proposition \ref{ns-equiv-def}, $A \ns_M N$.
\end{proof}

A similar result holds for $\nes$, see Lemma \ref{uniq-nsa}. 

The following consequence of invariance will be used repeatedly:

\begin{lem}\label{right-existence}
  Assume $\nf$ satisfies $(E_1)_M$. Assume $A \nf_M N$, and $N' \succ N$. Then there is $N'' \equiv_{N} N'$ such that $A \nf_M N''$.
\end{lem}
\begin{proof}
  By $(E_1)_M$, there is $f: A' \equiv_N A$, $A' \nf_M N'$. Thus $f: (A', N') \equiv_N (A, f[N'])$, so letting $N'' := f[N']$ and applying invariance, we obtain $A \nf_M N''$.
\end{proof}

Even though we will not use it, we note that an analogous result holds for left extension, see Lemma \ref{left-ext}.

We now would like to strengthen Lemma \ref{right-existence} as follows: suppose we are given $A$, $M \prec N_0 \prec N$, and assume $N$ is ``very big'' (e.g. it is $\left(2^{|A| + \| N_0\| }\right)^+$-saturated), but does \emph{not} contain $A$. Can we find $N_0' \equiv_M N_0$ with $A \nf_M N_0'$, \emph{and} $N_0' \prec N$?

We give this property a name: 

\begin{mydef}[Strong extension]\label{eplus-def}
  An independence relation $\nf$ has $(E_+)$ (strong extension) if for any $M \prec N_0$ and any set $A$, there is $N \succ N_0$ such that for all $N' \equiv_{N_0} N$, there is $N_0' \equiv_M N_0$ with $A \nf_M N_0'$ and $N_0' \prec N'$.
\end{mydef}

Intuitively, $(E_+)$ says that no matter which isomorphic copy $N'$ of $N$ we pick, even if $N'$ does not contain $A$, $N'$ is so big that we can still find $N_0'$ inside $N'$ with the right property. This is stronger than $(E)$ in the following sense:

\begin{prop}
  If $\nf$ has $(E_+)$, $\nf$ has $(E_0)$. If in addition $\nf$ has $(T_\ast)$, then $\nf$ has $(E_1)$. Thus if $\nf$ has $(E_+)$ and $(T_\ast)$, it has $(E)$.
\end{prop}
\begin{proof}
  Use monotonicity and Remark \ref{ext-conseq}.
\end{proof}
\begin{remark}
Example \ref{counterexample} shows $(E_+)$ does \emph{not} follow from $(E)$.
\end{remark}

Strong extension allows us to prove canonicity:

\begin{lem}\label{eplus-canon}
   Assume $(E_1^{(1)})_M, (E_+^{(2)})_M$. Assume also that $\1nf_M \subseteq \ns_M$.

   Then $\1nf_M \subseteq \2nf_M$.
\end{lem}
\begin{proof}
  Assume $A \1nf_M N_0$. We show $A \2nf_M N_0$. Fix $N \succ N_0$ as described by $(E_+^{(2)})_M$.  By Lemma \ref{right-existence}, we can find $N' \equiv_{N_0} N$ such that $A \1nf_M N'$. By definition of $N$, one can pick $N_0' \equiv_M N_0$ with $N_0' \prec N'$ and $A \2nf_M N_0'$. 

  We have $A \ns_M N'$, $M \prec N_0', N_0 \prec N'$, and $N_0' \equiv_M N_0$, so by definition of nonsplitting, $N_0' \equiv_{AM} N_0$. By invariance, $A \2nf_M N_0$, as needed.
\end{proof}

\begin{cor}[Canonicity of forking from strong extension]\label{eplus-cor}
  Assume:

  \begin{itemize}
  \item $(U^{(1)})_M, (E^{(1)})_M$.
  \item $(U^{(2)})_M, (E_+^{(2)})_M$. 
  \end{itemize}

  Then $\1nf_M = \2nf_M$.
\end{cor}
\begin{proof}
  By Lemma \ref{onedir}, it is enough to see $\1nf_M \subseteq \2nf_M$. By Lemma \ref{nsmon}, $\1nf_M \subseteq \ns_M$. The result now follows from Lemma \ref{eplus-canon}.
\end{proof}

We now proceed to show that $(E_+)$ follows from $(E)$, $(T_\ast)$, $(S)$ and $(L)$. We will use the following important concept:

\begin{mydef}[Independent sequence]
  Let $I$ be a linearly ordered set. A sequence of sets $(A_i)_{i \in I}$ is \emph{independent} over a model $M$ if there is a strictly increasing continuous chain of models $(N_i)_{i \in I}$ such that for all $i \in I$:

  \begin{enumerate}
    \item $M \cup \bigcup_{j < i} A_j \subseteq N_i$ and $N_0 = M$.
    \item $A_i \nf_M N_i$.
  \end{enumerate}
\end{mydef}

This generalizes the notion of independent sequence from the first-order case. The most natural definition would only require $A_i \nfm_M \bigcup_{j < i} A_j$ (for some closure $\nfm$ of $\nf$) but it turns out it is convenient to have a sequence of models $(N_i)_{i \in I}$ witnessing the independence in a uniform way.

We note that very similar definitions appear already in the litterature. See \cite[Definition 3.2]{jasi}, \cite[Section III.5]{shelahaecbook}, or \cite[Definition V.D.3.15]{shelahaecbook2}.

Just like in the first-order case, the extension property allows us to build independent sequences:

\begin{lem}[Existence of independent sequences]\label{indep-exist}
  Assume $(E)_M$. Let $A$ be a set, and let $\delta$ be an ordinal. Then there is a sequence $(A_i)_{i < \delta}$ independent over $M$ so that $A_i \equiv_M A$ for all $i < \delta$, and $A_0 = A$.
\end{lem}
\begin{proof}
  Define the $(A_i)_{i < \delta}$ and the $(N_i)_{i < \delta}$ witnessing the independence of the sequence by induction on $i < \delta$. Take $N_0 = M$ and $A_0 = A$. Assume inductively $(A_j)_{j < i}, (N_j)_{j < i}$ have been defined. If $i$ is a limit, let $N_i := \bigcup_{j < i} N_j$. If $i$ is a successor, let $N_i$ be any model containing $M \cup \bigcup_{j < i} (A_j \cup N_j)$ and strictly extending the previous $N_j$'s. By $(E)_M$, there is $A_i \equiv_M A$ such that $A_i \nf_M N_i$. Thus $(A_i)_{i < \delta}$ is as desired.
\end{proof}

The next result is key to the proof of $(E_+)$. It is adapted from \cite[Theorem II.2.18]{baldwinbook}.

\begin{lem}\label{baldwin-218}
  Assume $\nf$ has $(S), (T_\ast)_M, (L)$. Let $A$ be a set, and let $\mu := \kappa_{|A|} (\nf)$.  Then whenever $(M_i)_{i < \mu}$ is an independent sequence over $M$ with $M \prec M_i$ for all $i$, there is $i < \mu$ with $A \nf_M M_i$.
\end{lem}
\begin{proof}
  Let $(N_i)_{i < \mu}$ witness independence of the $M_i$'s. Let $N_\mu := \bigcup_{i < \mu} N_i$. By definition of $\mu$, there is $i < \mu$ so that $A \nf_{N_i} N_\mu$. By $(S)$, there is a model $N_A$ with $N_i \prec N_A$, $A \subseteq N_A$, and $N_\mu \nf_{N_i} N_A$. By $(M)$, $M_i \nf_{N_i} N_A$. Since the $M_i$'s are independent, we also have $M_i \nf_M N_i$. By $(T_\ast)_M$, $M_i \nf_M N_A$. By $(S)$ (recall that $M \prec M_i$), $N_A \nf_M M_i$. By $(M)$, $A \nf_M M_i$, as desired.
\end{proof}

\begin{remark}
  The same proof works if we replace $\nf$ by its minimal closure $\overline{\nf}$, and $(M_i)_{i < \mu}$ by an arbitrary sequence $(B_i)_{i < \mu}$ independent over $M$.
\end{remark}

\begin{cor}\label{eplus}
  Assume $(E)_M$, $(S), (T_\ast)_M$, and $(L)$. Then $(E_+)_M$.
\end{cor}
\begin{proof}
  Fix $A$ and $N_0 \succ M$. Let $\mu := \kappa_{|A|} (\nf)$. By Lemma \ref{indep-exist}, there is a sequence $(M_i)_{i < \mu}$ independent over $M$ such that $M_i \equiv_M N_0$ for all $i < \mu$, and $M_0 = N_0$. Let $(N_i')_{i < \mu}$ witness independence of the $M_i$'s. We claim $N := \bigcup_{i < \mu} N_i'$ is as required. By construction, $N_0 = M_0 \prec N$.

  Now let $f: N \equiv_{N_0} N'$. Let $M_i' := f[M_i]$. Invariance implies $(M_i')_{i < \mu}$ is an independent sequence over $M$ inside $N'$, with $M_i' \equiv_M N_0$ for all $i < \mu$. By Lemma \ref{baldwin-218}, there is $i < \mu$ so that $A \nf_M M_i'$, so $N_0' := M_i'$ is exactly as needed.
\end{proof}

\begin{cor}\label{endcor}
  Assume:
  \begin{itemize}
    \item $(E^{(1)})_M, (U^{(1)})_M$.
    \item $(E^{(2)})_M, (U^{(2)})_M, (L^{(2)}), (S^{(2)}), (T_\ast^{(2)})_M$.
  \end{itemize}

  Then $\1nf_M = \2nf_M$.  
\end{cor}
\begin{proof}
  Combine Corollaries \ref{eplus-cor} and \ref{eplus}.
\end{proof}

We will see (Corollary \ref{sym-trans}) that $(S)$ and $(T_\ast)$ follow from $(E)$, $(U)$, and $(L)$. We now argue that the other hypotheses are necessary. The following example (versions of which appears at various places in the literature, e.g. \cite[Example II.6.4]{shelahaecbook}, \cite[Example 6.6]{Adler}) shows we cannot remove the local character assumption from Corollary \ref{endcor}. In particular, $(E_+)$ does not follow from $(E)$ and $(U)$ alone. The example also shows the $\text{AxFr}_3$ framework (see \cite[Definition V.B.1.9]{shelahaecbook2}) is not canonical.

\begin{example}\label{counterexample}
  Let $T_{ind}$ be the first-order theory of the random graph, and let $K$ be the class of models of $T_{ind}$, ordered by first-order elementary substructure. Define

  \begin{itemize}
    \item $A \1nf_M N$ iff $A \cap N \subseteq M$, and there are no edges between $A \backslash M$ and $N \backslash M$.
    \item $A \2nf_M N$ iff $A \cap N \subseteq M$, and all the possible cross edges between $A \backslash M$ and $N \backslash M$ are present.
  \end{itemize}

  It is routine to check that both $\1nf$ and $\2nf$ are independence relations with $(E)$, $(U)$, $(S)$, $(T)$, $(T_\ast)$, $(C)_{\aleph_0}$. Yet $\1nf \neq \2nf$, so one knows from Corollary \ref{endcor} (or from first-order stability theory) that $K$ can have \emph{no} independence relation which in addition has $(L)$ or $(E_+)$.

Of course, $T_{ind}$ is simple, so first-order nonforking will actually have $(E_+)$, local character, transitivity and symmetry (but not uniqueness).

  A concrete reason $(E_+)$ does not hold e.g. for $\1nf$ is that given $M \prec N_0$ one can pick $a \not \in N_0$ such that there is an edge from $a$ to any element of $N_0$. Then for any $N \succ N_0$, one can again pick $N' \equiv_{N_0} N$, disjoint from $\{a\} \cup (N \backslash N_0)$ such that there is an edge from $a$ to any element of $N'$. Then $a \1nf_M N_0'$ for $N_0' \prec N$ implies $N_0' = M$. Local character fails for a similar reason.
\end{example}

\begin{example}\label{counterexample2}
It is also easy to see that $(E^{(2)})$ and $(U^{(2)})$ are necessary in Corollary \ref{endcor}. Assume $\1nf$ has $(E)$, $(U)$, $(S), (T_\ast)$, and $(L)$. Then the independence relation $\2nf$ defined by $A \2nf_M N$ for all $A$ and $M \prec N$ satisfies $(E), (S), (T_\ast), (L)$, but not $(U)$, so is distinct from $\1nf$. 

Similarly define $A \2nf_M N$ if and only if $M \prec N$ and either both $A \1nf_M N$ and $\| M\|  \ge \text{LS} (K)^+$, or $M = N$. Then $\2nf$ has $(E_0)$, $(U)$, $(S), (T_\ast)$ and $(L)$, but does not have $(E_1)_M$ if $M$ is a model of size $\text{LS} (K)$. This last example was adapted from \cite[Example 6.4]{Adler}.
\end{example}
\begin{remark}
  After the initial submission of this paper, it was shown in \cite[Lemma 9.1]{indep-aec-v5} that $(E)$ can be removed from the hypotheses of Corollary \ref{endcor} (but one has to replace it by $(C)_\kappa$) if one only wants the independence relations to agree over sufficiently saturated models.
\end{remark}

\section{Relationship between various properties}\label{relations-properties}

In this section, we investigate some of relations between the properties introduced earlier. We first discuss the interaction between properties of an independence relation and properties of its closures, and show how to obtain transitivity from various other properties. We then show how to obtain symmetry from existence, extension, uniqueness, and local character (Corollary \ref{sym-trans}). This second part has a stability-theoretic flavor and most of it does not depend on the first part.

Most of the material in the first part of this section is not used in the rest of the paper, but the concept of closure (Definition \ref{closuredef}) felt unmotivated without it. Our investigation remains far from exhaustive, and leaves a lot of room for further work.

\subsection{Properties of the minimal closure}

Recall the notion of closure of an independence relation (Definition \ref{closuredef}). We would like to know when we can transfer properties from an independence relation to its closures and vice-versa.

For an arbitrary closure, we can say little:

\begin{lem}\label{nfm-arbitrary}
  Let $\nfm$ be a closure of $\nf$. Then:

  \begin{enumerate}
    \item A property in the following list holds for $\nf$ if and only if it holds for $\nfm$: $(T_\ast)_M$, $(E_0)_M$, $(L)$.
    \item If a property in the following list holds for $\nfm$, then it holds for $\nf$: $(C)_\kappa$, $(T)_M$, $(E_1)_M$, $(U)_M$.
  \end{enumerate}
\end{lem}
\begin{proof} \
  \begin{enumerate}
    \item Because those properties have the same definition for $\nf$ and $\nfm$.
    \item Straightforward from the definitions.
  \end{enumerate}
\end{proof}

The minimal closure is more interesting. We start by generalizing Lemma \ref{right-existence}:

\begin{lem}\label{right-existence-nfm}
  Assume $\nf$ satisfies $(E_1)_M$. Let $\nfm$ be the minimal closure of $\nf$. Assume $A \nfm_M C$, and let $B$ be an arbitrary set. Then there is $B' \equiv_{MC} B$ such that $A \nfm_M B'$.
\end{lem}
\begin{proof}
  Let $N$ be a model containing $C$ and $M$ such that $A \nf_M N$. Let $N'$ be a model containing $NB$. By Lemma \ref{right-existence}, there is $N'' \equiv_N N'$ such that $A \nf_M N''$. Now use monotonicity to get the result.
\end{proof}

The next lemma tells us that the minimal closure is the only one that will keep the extension property:

\begin{lem}\label{min-arbitrary-closure}
  Let $\nfm$ be a closure of $\nf$ and let $\overline{\nf}$ be the minimal closure of $\nf$. Assume $\nf$ has $(E_1)_M$. Then $\nfm_M = \overline{\nf_M}$ if and only if $\nfm$ has $(E_1)_M$.
\end{lem}
\begin{proof}
  Assume first $\nfm_M = \overline{\nf_M}$. Let $C \subseteq C'$, and assume $A \nfm_M C$. Then by definition of the minimal closure, there exists $N \succ M$ containing $C$ such that $A \nf_M N$. Let $N'$ be a model containing $N$ and $C'$. By $(E_1)_M$ for $\nf$, there is $A' \equiv_N A$ so that $A' \nf_M N'$. By monotonicity, $A' \nfm_M C'$, and since $N$ contains $C$, $A' \nfm_M C$.

  Conversely, assume $\nfm_M$ has $(E_1)_M$. We know already that $\overline{\nf} \subseteq \nfm$, so assume $A \nfm_M C$. Let $N$ be a model containing $M$ and $C$. By Lemma \ref{right-existence-nfm}, there is $N' \equiv_{MC} N$ so that $A \nfm_M N'$, so $A \overline{\nf_M} C$, as needed.
\end{proof}

\begin{lem}\label{inv-nfm}
  Let $\nfm$ be the \emph{minimal} closure of $\nf$. Then
  \begin{enumerate}
    \item $(E)_M$ holds for $\nf$ if and only if it holds for  $\nfm$.
    \item $(S)_M$ holds for $\nf$ if and only if it holds for  $\nfm$.
    \item If $\nf$ has $(E)_M$, then it has $(U)_M$ if and only if $\nfm$ does.
    \item If $\nf$ has $(E)$, then it has $(T)$ if and only if $\nfm$ does\footnote{More precisely, if $\nf$ has $(E)_{M_1}$, and for $M_0 \prec M_1 \prec M_2$, we have that $M_2 \nf_{M_1} N$, $M_1 \nf_{M_0} N$ implies $M_2 \nf_{M_0} N$, then $M_2 \nfm_{M_1} C$, $M_1 \nfm_{M_0} C$ implies $M_2 \nfm_{M_0} C$.}.
  \end{enumerate}
\end{lem}
\begin{proof} \
  \begin{enumerate}
  \item By Lemmas \ref{nfm-arbitrary} and \ref{min-arbitrary-closure}.
  \item Straightforward from the definition of symmetry and monotonicity.
  \item One direction holds by Lemma \ref{nfm-arbitrary}. For the other direction, assume $\nfm$ has $(E_1)_M$ and $\nf$ has $(U)_M$. Assume $A \nfm_M C$ and $A' \nfm_M C$, with $f: A \equiv_M A'$. Let $N$ be a model containing $MC$ such that $A \nf_M N$. By extension again, find $h: A' \equiv_{MC} A''$ such that $A'' \nf_M N$. We know $h' := h \circ f : A \equiv_M A''$, so by uniqueness, there is $h'': A \equiv_{N'} A''$, and $h'' \upharpoonright A = h' \upharpoonright A = (h \circ f) \upharpoonright A$, so $f \upharpoonright A = (h^{-1} \circ h'')  \upharpoonright A$. Therefore $g := h^{-1} \circ h''$ is the desired witness that $A \equiv_{MC} A'$.
  \item\label{trans-nfm} Let $M_0 \prec M_1 \prec M_2$, and assume $M_1 \nfm_{M_0} C$, $M_2 \nfm_{M_1} C$. Let $N$ be an extension of $M_1$ containing $C$ such that $M_2 \nfm_{M_1} N$. Let $\chi$ be a big cardinal, so that $(V_\chi, \in)$ reflects enough set theory and contains $N M_2$. Let $N'$ be what $V_\chi$ believes is the monster model. 

    By Lemma \ref{right-existence}, there is $f: N'' \equiv_N N'$ such that $M_2 \nf_{M_1} N''$. Notice that $C M_1 \subseteq N \subseteq N'$, so since we took $\chi$ big enough, we can apply the definition of the minimal closure \emph{inside} $V_\chi$ to get $N_0' \prec N'$ containing $M_0$ and $C$ so that $M_1 \nf_{M_0} N_0'$. Let $N_0 := f[N_0']$. By invariance, $M_1 \nf_{M_0} N_0$, and $N_0 \prec N''$, so by monotonicity, $M_2 \nf_{M_1} N_0$, so by $(T)_{M_0}$ for $\nf$, $M_2 \nf_{M_0} N_0$. By monotonicity again, $M_2 \nfm_{M_0} C$.
  \end{enumerate}
\end{proof}

The following remains to be investigated:

\begin{question}\label{min-closure-q}
  Let $\nfm$ be the minimal closure of $\nf$. Under what conditions does $(C)_\kappa$ for $\nf$ imply $(C)_\kappa$ for $\nfm$?
\end{question}

We can use Lemma \ref{inv-nfm} to prove a variation on Lemma \ref{nsmon}.

\begin{lem}\label{uniq-nsa}
  Assume $\nf$ has $(E)_M$ and $(U)_M$. Then $\nf_M \subseteq \nes_M$.
\end{lem}
\begin{proof}
  Let $\nfm$ be the minimal closure of $\nf$. By Lemma \ref{inv-nfm}, $\nfm$ has $(E)_M$ and $(U)_M$. 

  Assume $A \nf_M N$. Let $C_1, C_2 \subseteq N$, and $h: C_1 \equiv_M C_2$. By monotonicity, $A \nfm_M C_\ell$ for $\ell = 1, 2$. By invariance, $h [A] \nfm_M C_2$. By $(U)_M$, there is $f: A \equiv_{MC_2} h[A]$ with $f \upharpoonright A = h \upharpoonright A$. By (the proof of) Proposition \ref{ns-equiv-def}, $A \nes_M N$.
\end{proof}

\begin{question}
  Is the $(E)_M$ hypothesis necessary?
\end{question}

We can also obtain a left version of Lemma \ref{right-existence}:

\begin{lem}\label{left-ext}
  Let $\nfm$ be a closure of $\nf$. Assume $\nf$ has $(E)_N$, and $\nfm$ has $(T)_{M_1}$. Suppose that $N \nf_{M_1} M_2$, with $N \succ M_1$. Then for all $N' \succ N$, there exists $N'' \equiv_N N'$ such that $N'' \nf_{M_1} M_2$.

  In particular, this holds if $\nf$ has $(E)$ and $(T)$.
\end{lem}
\begin{proof}
  The last line follows from part (\ref{trans-nfm}) of Lemma \ref{inv-nfm} by taking $\nfm$ to be the minimal closure of $\nf$.

  To see the rest, let $N_3$ be a model containing $M_2N$. By $(E)_N$, there is $N'' \equiv_N N'$ such that $N'' \nf_{N} N_3$. Since $M_2 \subseteq N_3$, $N'' \nfm_{N} M_2$. By hypothesis, $N \nfm_{M_1} M_2$. So since $\nfm$ has $(T)_{M_1}$, $N'' \nfm_{M_1} M_2$. Since $M_2 \succ M_1$, $N'' \nf_{M_1} M_2$.
\end{proof}

Finally, we can also use symmetry to translate between the transitivity properties:

\begin{lem}\label{right-trans-sym}
  Assume $\nf$ has $(S)$. Then:
  
  \begin{enumerate}
    \item If $\nf$ has $(T_\ast)_{M_0}$, then $\nf$ has $(T)_{M_0}$.
    \item If $\nf$ has $(T)_{M_0}$ and $(E)$, then it has $(T_\ast)_{M_0}$.
  \end{enumerate}
\end{lem}
\begin{proof} Let $M_0 \prec M_1 \prec M_2$. Let $\nfm$ be the minimal closure of $\nf$. By Lemma \ref{inv-nfm}, $\nfm$ has $(S)$.
  \begin{enumerate}
    \item By Lemma \ref{nfm-arbitrary}, $\nfm$ has $(T_\ast)_{M_0}$. Now use symmetry.
    \item By part (\ref{trans-nfm}) of Lemma \ref{inv-nfm}, $\nfm$ has $(T)_{M_0}$. Now use symmetry.
  \end{enumerate}
\end{proof}

This gives us one way to obtain right transitivity for coheir:

\begin{cor}\label{coheir-right-trans}
  Assume $\ch$ has $(S)$ and $(E)$. Then $\ch$ has $(T_\ast)$.
\end{cor}
\begin{proof}
  By Proposition \ref{elem-prop-ch-ns}, $\ch$ has $(T)$. Apply Lemma \ref{right-trans-sym}.
\end{proof}

Another way to obtain right transitivity from other properties appears in \cite[Claim II.2.18]{shelahaecbook}:

\begin{lem}\label{ext-uniq-trans}
  Assume $\nf$ has $(E_1)_M$ and $(U)$. Then $\nf$ has $(T_\ast)_M$.
\end{lem}
\begin{proof}
  Let $M_0 \prec M_1 \prec M_2$, and assume $A \nf_{M_0} M_1$ and $A \nf_{M_1} M_2$. By $(E_1)_M$, there exists $A' \equiv_{M_1} A$ such that $A' \nf_{M_0} M_2$. By base monotonicity, $A' \nf_{M_1} M_2$. By uniqueness, $A \equiv_{M_2} A'$. By invariance, $A \nf_{M_0} M_2$.
\end{proof}

\subsection{Getting symmetry}

We prove that symmetry follows from $(E)$, uniqueness and local character and deduce the main theorem of this paper (Corollary \ref{endcor-improved}). We start by assuming some stability. The following is a strengthening of unstability that is sometimes more convenient to work with:

\begin{mydef}
  Let $\alpha$ and $\lambda$ be cardinals. $K$ has the \emph{$\alpha$-order property of length $\lambda$} if there is a sequence $(\bar{a}_i)_{i < \lambda}$ of tuples, with $\ell (\bar{a}_i) = \alpha$, so that for any $i_0 < j_0 < \lambda$ and $i_1 < j_1 < \lambda$, $\bar{a}_{i_0} \bar{a}_{j_0} \not \equiv \bar{a}_{j_1} \bar{a}_{i_1}$.

  $K$ has the \emph{$\alpha$-order property} if it has the $\alpha$-order property of all lengths.

  $K$ has the \emph{order property} if it has the $\alpha$-order property for some cardinal $\alpha$.
\end{mydef}

This is a variation on the order property defined in \cite{sh394} taken from \cite{tamenessone}. It is stronger than unstability:

\begin{fact}\label{op-unstable}
  Let $\alpha$ be a cardinal. If $K$ has the $\alpha$-order property, then $K$ is $\alpha$-unstable.
\end{fact}
\begin{proof}[Proof sketch]
  This is \cite[Claim 4.7.2]{sh394}. Shelah's proof is ``Straight.'',  so we elaborate a little. 

  Let $\lambda \ge \text{LS} (K)$. We show $K$ is $\alpha$-unstable in $\lambda$. Let $I \subseteq \hat{I}$ be linear orderings such that $\| I\|  \le \lambda$, $\| \hat{I}\|  > \lambda$, and $I$ is dense in $\hat{I}$. Combining Shelah's presentation theorem with Morley's method, we can get a sequence $\hat{\mathbf{I}} := \left<\bar{a}_i \mid i \in \hat{I}\right>$ with $\ell (\bar{a}_i) = \alpha$ and $i_0 < j_0$, $i_1 < j_1$ implies $\bar{a}_{i_0} \bar{a}_{j_0} \not \equiv \bar{a}_{j_1} \bar{a}_{i_1}$. Let $\mathbf{I} := \left< \bar{a}_i \mid i \in I \right>$.

Now for any $i < j$ in $\hat{I}$, $\bar{a}_i \not \equiv_{\mathbf{I}} \bar{a}_j$. Indeed, pick $i < k < j$ with $k \in I$. Then $\bar{a}_i \bar{a}_k \not \equiv \bar{a}_j \bar{a}_k$ by construction, so $\bar{a}_i \not \equiv_{\bar{a}_k} \bar{a}_j$. This completes the proof that $K$ is $\alpha$-unstable in $\lambda$.
\end{proof}

We are now ready to prove symmetry. The argument is similar to  \cite[Theorem III.4.13]{shelahfobook} or \cite[Theorem 5.1]{sh48}.

\begin{thm}[Symmetry]\label{symmetry}
  Assume $\nf$ has $(E)_M$ and $\nf_M \subseteq \nes_M$. Assume in addition that $K$ does not have the order property. Then $\nf$ has $(S)_M$.
\end{thm}
\begin{proof}
  Let $\nfm$ be the minimal closure of $\nf$. Recall that by Lemma \ref{inv-nfm}, $\nf$ has $(S)_M$ if and only if $\nfm$ has $(S)_M$.
  
 Assume for a contradiction $\nf$ does not have $(S)_M$. Pick $A$ and $M \prec N$ such that $A \nf_M N$, but $N \nnfm_M A$. Let $\lambda$ be an arbitrary uncountable cardinal. We will show that $K$ has the $(\| N\|  + |A|)$-order property of length $\lambda$. This will contradict the assumption that $K$ does not have the order property.

We will build increasing continuous $\seq{M_\alpha \in K: \alpha < \lambda}$, and $\seq{A_\alpha, M_\alpha', N_\alpha : \alpha < \lambda}$ by induction so
\begin{enumerate}
	\item $M_0 \succ N$ and $A \subseteq |M_0|$.
	\item $N_\alpha \equiv_M N$ and $N_\alpha \prec M_\alpha'$.
	\item $A_\alpha \equiv_N A$ and $A_\alpha \subseteq M_{\alpha+1}$.
	\item $M_\alpha \prec M'_\alpha \prec M_{\alpha+1}$.
	\item $N_\alpha \nf_M M_\alpha$ and $A_\alpha \nf_M M'_\alpha$.
\end{enumerate}

\paragraph{This is possible}

Let $M_0$ be any model containing $AN$. At $\alpha$ limits, let $M_\alpha := \bigcup_{\beta < \alpha} M_\beta$. Now assume inductively that $M_\beta$ has been defined for $\beta \le \alpha$, and $A_\beta$, $N_\beta$, $M_\beta'$ have been defined for $\beta < \alpha$. Use $(E)_M$ to find $N_\alpha \equiv_M N$ with $N_\alpha \nf_M M_\alpha$. Now pick $M_\alpha' \ge M_\alpha$ containing $N_\alpha$. Now, by $(E)_M$ again, find $A_\alpha \equiv_N A$ with $A_\alpha \nf_M M_\alpha'$. Pick $M_{\alpha + 1} \succ M_{\alpha}$ containing $A_\alpha$ and $M_\alpha'$.

\paragraph{This is enough}

We show that for $\alpha, \beta < \lambda$:

\begin{enumerate}
  \item If $\beta < \alpha$, $(A,N) \not \equiv_M (A_\beta, N_\alpha)$.
  \item If $\beta \ge \alpha$, $(A, N) \equiv_M (A_\beta, N_\alpha)$.
\end{enumerate}

For (1), suppose $\beta < \alpha$.  Since $A \subseteq M \prec M_\alpha$, we have $N_\alpha \nfm_M A$.  Then we can use the invariance of $\nfm$ and the assumption of no symmetry to conclude $(A, N_\alpha) \not \equiv_M (A,N)$.  On the other hand, we know that $N_\alpha \nes_M M_\alpha$.  Since $A, A_\beta \subseteq M_\alpha$ and  $A \equiv_M A_\beta$, we must have $(A,N_\alpha) \equiv_M (A_\beta, N_\alpha)$.  Thus $(A, N) \not \equiv_M (A_\beta, N_\alpha)$.

To see (2), suppose $\beta \geq \alpha$ and recall that $(A,N) \equiv_M (A_\beta, N)$.  We also have that $A_\beta \nes_M M'_{\beta}$. $N \equiv_M N_\alpha$ and $N, N_\alpha \subseteq M_\beta'$, the definition of non explicit splitting implies that $(A_\beta, N) \equiv_M (A_\beta, N_\alpha)$.  This gives us that $(A, N) \equiv_M (A_\beta, N_\alpha)$ as desired.

\end{proof}

\begin{remark}
  The same proof can be used to obtain symmetry in the good frame framework. This is used in the construction of a good frame of \cite{ss-tame-toappear-v3}.
\end{remark}

\begin{cor}\label{ext-uniq-sym}
  Assume $K$ does not have the order property. Assume $\nf$ has $(E)_M$ and $(U)_M$. Then $\nf$ has $(S)_M$.
\end{cor}
\begin{proof}
  By Lemma \ref{uniq-nsa}, $\nf_M \subseteq \nes_M$. Now apply Theorem \ref{symmetry}.
\end{proof}

If in addition we assume local character, we obtain the ``no order property'' hypothesis:

\begin{lem}\label{local-nop}
  Assume $\nf$ has $(U)$ and $(L)$ (or just $\kappa_1 (\nf) < \infty$). Then $K$ is $\alpha$-stable for all $\alpha$. In particular, it does not have the order property.
\end{lem}
\begin{proof}
  That $\alpha$-stability implies no $\alpha$-order property is the contrapositive of Fact \ref{op-unstable}. Now, assume $(U)$ and let $\mu := \kappa_1 (\nf) < \infty$. Fix a cardinal $\alpha \ge 1$. We want to see $K$ is $\alpha$-stable. By Remark \ref{monot-stab}, we can assume without loss of generality $\alpha \ge \mu + \text{LS} (K)$.

  Let $\lambda := \beth_{\alpha^+}$. Then:

  \begin{enumerate}
    \item $\lambda$ is strong limit.
    \item $\cf (\lambda) = \alpha^+ > \mu + \text{LS} (K)$.
    \item $\lambda^\alpha = \sup_{\gamma < \lambda} \gamma^\alpha = \lambda$.
  \end{enumerate}

  We claim that $K$ is $\alpha$-stable in $\lambda$. By Fact \ref{stab-longtypes}, it is enough to see it is $1$-stable in $\lambda$. Suppose not. Then there exists $M \in K_\lambda$, and $\{a_i\}_{i < \lambda^+}$ such that $i < j$ implies $a_i \not \equiv_M a_j$. Let $(M_i)_{i < \lambda}$ be increasing continuous such that $M = \bigcup_{i < \lambda} M_i$ and $\| M_i\|  < \lambda$. By definition of $\mu$, for each $i < \lambda^+$, there exists $k_i < \lambda$ such that $a_i \nf_{M_{k_i}} M$. By the pigeonhole principle, we can shrink $\{a_i\}_{i < \lambda^+}$ to assume without loss of generality that $k_i = k_0$ for all $i < \lambda^+$. Since there are at most $2^{\| M_{k_0}\| } < \lambda$ many types over $M_{k_0}$, there exists $i < j < \lambda^+$ such that $a_i \equiv_{M_{k_0}} a_j$. By uniqueness, $a_i \equiv_M a_j$, a contradiction.
\end{proof}
\begin{cor}\label{sym-trans}
  Assume $\nf$ has $(E)_M$, $(U)$ and $(L)$ (or just $\kappa_1 (\nf) < \infty$). Then $\nf$ has $(S)_M$ and $(T_\ast)_M$.
\end{cor}
\begin{proof}
  Lemma \ref{ext-uniq-trans} gives $(T_\ast)_M$. Combine Lemma \ref{local-nop} and Corollary \ref{ext-uniq-sym} to obtain $(S)_M$.
\end{proof}

Thus we obtain another version of the canonicity theorem:

\begin{cor}[Canonicity of forking]\label{endcor-improved}
  Let $\1nf$ and $\2nf$ be independence relations. Assume:
  \begin{itemize}
    \item $(E^{(1)})_M, (U^{(1)})_M$.
    \item $(E^{(2)})_M, (U^{(2)}), (L^{(2)})$.
  \end{itemize}

  Then $\1nf_M = \2nf_M$.  

  In particular, there can be at most one independence relation satisfying existence, extension, uniqueness, and local character.
\end{cor}
\begin{proof}
  Combine Corollaries \ref{endcor} and \ref{sym-trans}.
\end{proof}

\section{Applications}\label{last-section}

\subsection{Canonicity of coheir}

Fix a regular $\kappa > \text{LS} (K)$. Below, when we say coheir has a given property, we mean that it has that property for base models in $K_{\ge \kappa}$.

We are almost ready to show that coheir is canonical, but we first need to show it has local character. We will use the following strengthening that deals with subsets instead of chains of models:

\begin{mydef}
  Let $\nf$ be an independence relation. For $\alpha$ a cardinal, let $\bkappa_\alpha = \bkappa_\alpha (\nf)$ be the smallest cardinal such that for all $N$, and all $A$ with $|A| = \alpha$, there exists $M \le N$ with $\| M\|  < \bkappa_\alpha$ and $A \nf_M N$. $\bkappa_\alpha = \infty$ if there is no such cardinal.
\end{mydef}

\begin{remark}
  For all $\alpha$, $\kappa_\alpha(\nf) \leq \bkappa_\alpha(\nf)^+$. Thus $\bkappa_\alpha (\nf) < \infty$ implies $\kappa_\alpha (\nf) < \infty$.
\end{remark}
\begin{remark}
  The converse is also true: under some reasonable hypotheses, $\kappa_\alpha(\nf) < \infty$ implies $\bkappa_\alpha(\nf) < \infty$. This appears as \cite[Proposition 2.15]{bv-sat-v3} (circulated after the initial submission of this paper).
\end{remark}

\begin{thm}[Local character for coheir]\label{nflocal}
  Assume $\ch$ has $(S)$. Then $\bkappa_{\alpha} (\ch) \le ((\alpha + 2)^{<\kappa})^+$. In particular, $\ch$ has $(L)$.
\end{thm}

The proof is similar to that of \cite[Theorem 1.6]{adler-rank}\footnote{After proving the result, we noticed that a similar argument also appears in the proof of $(B)_\mu$ in \cite[Proposition 4.8]{makkaishelah}.}. The key is that $\chm$ always satisfies a dual to local character:

\begin{lem}\label{ch-dual-loc}
  Let $N, C$ be given. Then there is $M \le N$, $\| M\|  \le (|C| + 2)^{<\kappa} + LS(K)$ such that $N \chm_M C$.
\end{lem}
\begin{proof}[Proof sketch]
  For each of the $|C|^{<\kappa}$ small subsets of $C$, look at the $\le 2^{<\kappa}$ small types over that set (realized in $N$), and collect a realization of each in a set $A \subseteq |N|$.  Then pick $M \prec N$ to contain $A$ and be of the appropriate size.
\end{proof}

We will also use the following application of the fact $\chm$ has $(C)_\kappa$ and a strong form of base monotonicity.

\begin{lem}\label{ch-cont-chain}
Let $\lambda$ be such that $\cf \lambda \geq \kappa$. Let $(A_i)_{i < \lambda}$, $(M_i)_{i < \lambda}$, $(C_i)_{i < \lambda}$ be (not necessarily strictly) increasing chains. Assume $A_i \chm_{M_i} C_i$ for all $i < \lambda$. Let $A_\lambda := \bigcup_{i < \lambda} A_i$, and define $M_\lambda$, $C_\lambda$ similarly. Then $A_\lambda \chm_{M_\lambda} C_\lambda$.
\end{lem}
\begin{proof}
  From the definition of $\chm$, we see that for all $i < \lambda$, $A_i \chm_{M_\lambda} C_i$. Now use the fact that $\chm$ has $(C)_\kappa$ (Proposition \ref{elem-prop-ch-ns}).
\end{proof}

\begin{proof}[Proof of Theorem \ref{nflocal}]
  Fix $\alpha$, and let $A$ and $N$ be given with $|A| = \alpha$. Let $\mu := (|A| + 2)^{<\kappa}$. Inductively build $(M_i)_{i \le \mu}, (N_i)_{i \le \mu}$ increasing continuous such that for all $i < \mu$:

  \begin{enumerate}
    \item $A \subseteq N_i$.
    \item $M_i \prec N$, $\| M_i\|  \le \mu$.
    \item $M_i \prec N_{i + 1}$.
    \item $N \chm_{M_{i + 1}} N_{i + 1}$.
  \end{enumerate}

  \paragraph{This is enough:}
  By König's lemma, $\cf{\mu} \ge \kappa$ so by Lemma \ref{ch-cont-chain}, $N \chm_{M_\mu} N_\mu$. Moreover, by (2), (3) and the chain axioms, $M_\mu \prec N_\mu, N$, and by (2), $\| M_\kappa\|  \le \mu$. Thus $N \ch_{M_\mu} N_\mu$, and one can apply $(S)$ to get $N_\mu \ch_{M_\mu} N$. By monotonicity, $A \ch_{M_\mu} N$, exactly as needed.
  \paragraph{This is possible:}
  Pick any $A \subseteq N_0$ with $\| N_0\|  \le \mu$ (this is possible since $\mu \ge |A| + \kappa > \text{LS} (K)$). Now, given $i$ non-limit, $(N_j)_{j \le i}$ and $(M_j)_{j < i}$, use Lemma \ref{ch-dual-loc} to find $M_i \prec N$, $\| M_i\|  \le (\| N_i\| )^{<\kappa} \le \mu$, such that $N \chm_{M_i} N_i$. Then pick any $N_{i + 1}$ extending both $M_i$ and $N_i$, with $\| N_{i + 1}\|  \le \mu$.

\end{proof}

We finally have all the machinery to prove:

\begin{thm}[Canonicity of coheir]\label{canon-coheir}
  Assume $K$ is fully $(<\kappa)$-tame, fully $(<\kappa)$-type short, and has no weak $\kappa$-order property\footnote{See \cite[Definition 4.2]{bg-v9}.}.

  Assume $\ch$ has $(E)$. Then:

  \begin{enumerate}
    \item $\ch$ has $(C)_\kappa$, $(T)$, $(T_\ast)$, $(S)$, $(U)$, and $(L)$.
    \item Any independence relation satisfying $(E)$ and $(U)$ must be $\ch$ (for base models in $K_{\ge \kappa}$).
  \end{enumerate}
\end{thm}
\begin{proof}
  By Proposition \ref{elem-prop-ch-ns}, $\ch$ has $(C)_\kappa$ and $(T)$. By Fact \ref{bg-facts}, $\ch$ has $(U)$ and $(S)$. By Corollary \ref{coheir-right-trans} (or Lemma \ref{ext-uniq-trans}), $\ch$ also has $(T_\ast)$. By Theorem \ref{nflocal}, $\ch$ has $(L)$. This takes care of (1). (2) follows from (1) and Corollary \ref{endcor}.
\end{proof}

\begin{cor}[Canonicity of coheir, assuming a strongly compact]
  Assume $\kappa$ is strongly compact, all models in $K_{\ge \kappa}$ are $\kappa$-saturated, and there exists $\lambda > \kappa$ such that $I (\lambda, K) < 2^\lambda$. Then:

  \begin{enumerate}
    \item $\ch$ has $(E)$, $(C)_\kappa$, $(T)$, $(T_\ast)$, $(S)$, $(U)$, and $(L)$.
    \item Any independence relation satisfying $(E)$ and $(U)$ must be $\ch$ (for base models in $K_{\ge \kappa}$).
  \end{enumerate}
\end{cor}
\begin{proof}
  By Proposition \ref{elem-prop-ch-ns}, $\ch$ has $(E_0)$. Thus by the moreover part of Fact \ref{bg-facts}, $\ch$ has $(E)$. Now apply Theorem \ref{canon-coheir}.
\end{proof}

\subsection{Canonicity of good frames}

As has already been noted, the framework $\text{AxFri}_3$ defined in \cite[Chapter V.B]{shelahaecbook2} is a precursor to our own, but Example \ref{counterexample} shows it is not canonical. Shelah also investigated an extension of $\text{AxFri}_3$ axiomatizing primeness (the ``primal framework'') but it is outside the scope of this paper.

We will however briefly discuss the canonicity of good frames. Good frames were first defined in \cite[Chapter II]{shelahaecbook}. We will assume the reader is familiar with their definition and basic properties. As already noted, the main difference with our framework is that a good frame is local: For a fixed $\lambda \ge \text{LS} (K)$, a good $\lambda$-frame assumes the existence of a nice independence relation $\nf$ where only $a \nf_M^{\bigN} N$ is defined, for $a$ an element of $\bigN$ and $M \prec N \prec \bigN$ models of size $\lambda$. 

In \cite[Section II.6]{shelahaecbook}, Shelah shows that, assuming a technical condition (that the frame is weakly successful), one can extend it uniquely to a \emph{non-forking frame}: basically an independence relation $\nf$ where $M_1 \nf_M^{\bigN} M_2$ is defined for $M \prec M_\ell \prec \bigN$ in $K_\lambda$, $\ell = 1, 2$. For the rest of this section, we fix $\lambda \ge \text{LS} (K)$ and we do \emph{not} assume the existence of a monster model (Hypothesis \ref{monster-model-hyp}). Recall however that the definition of a good frame implies $K_\lambda$ has some nice properties, i.e.\ it has amalgamation, joint embedding, no maximal model, is stable\footnote{Really only stable for basic types, but full stability follows (see \cite[Claim II.4.2.1]{shelahaecbook}).} in $\lambda$, and has a superlimit model.

\begin{fact}\label{nf-frame}
  If $\mathfrak{s}$ is a weakly successful good $\lambda$-frame, then it extends uniquely to a non-forking frame (i.e.\ using Shelah's terminology, there is a unique non-forking frame $\text{NF}$ that respects $\mathfrak{s}$).
\end{fact}
\begin{proof}
  Uniqueness is \cite[Claim II.6.3]{shelahaecbook} and existence is \cite[Conclusion II.6.34]{shelahaecbook}.
\end{proof}

As Shelah observed, Example \ref{counterexample} shows that a non-forking frame by itself need not be unique: we need to know it comes from a good frame, or at least that there is a good frame around. Shelah showed:

\begin{fact}\label{nf-frame-respect}
  Assume that $\mathfrak{s}$ is a $\text{good}^+$ $\lambda$-frame and $\text{NF}$ is a non-forking frame, both with underlying AEC $K$. Then $\text{NF}$ respects $\mathfrak{s}$. 
\end{fact}
\begin{proof}
  See \cite[Claim II.6.7]{shelahaecbook}.
\end{proof}

Here, $\text{good}^+$ is a technical condition asking for slightly more than just the original axioms of a good frame. 

We also have that a non-forking frame induces a good frame: 

\begin{fact}\label{nf-frame-respect2}
  Assume $K_\lambda$ has a superlimit, is stable in $\lambda$, and carries a non-forking frame $\text{NF}$ (so in particular it has amalgamation) with independence relation (defined for models in $K_\lambda$) $\nf$. Then the relation $a \nf_M^{\bigN} N$ holds iff there is $\bigN' \succ \bigN$ and $M \prec M' \prec \bigN'$ with $a \in M'$ so that $M' \nf_M^{\bigN'} N$ defines a type-full (i.e.\ the basic types are all the nonalgebraic types) good $\lambda$-frame $\mathfrak{t}$. If in addition $\text{NF}$ comes from a type-full weakly successful good $\lambda$-frame $\mathfrak{s}$, then $\mathfrak{s} = \mathfrak{t}$.
\end{fact}
\begin{proof}
  See \cite[Claim II.6.36]{shelahaecbook}.
\end{proof}

Thus we obtain the following canonicity result:

\begin{cor}
  Assume that $\mathfrak{s}_1$ is a weakly successful $\text{good}^+$ $\lambda$-frame and $\mathfrak{s}_2$ is a weakly successful good $\lambda$-frame in the same underlying AEC $K$. Assume further $\mathfrak{s}_1$ and $\mathfrak{s}_2$ are type-full (i.e.\ their basic types are all the nonalgebraic types). Then $\mathfrak{s}_1 = \mathfrak{s}_2$.  
\end{cor}
\begin{proof}
  Using Fact \ref{nf-frame}, let $\text{NF}_\ell$ be the non-forking frame extending $\mathfrak{s}_\ell$ for $\ell = 1,2$.
  By Fact \ref{nf-frame-respect}, $\text{NF}_2$ respects $\mathfrak{s}_1$, so $\text{NF}_1 = \text{NF}_2$. By Fact \ref{nf-frame-respect2} (the existence of a good frame implies the stability and superlimit hypotheses), we must also have $\mathfrak{s}_1 = \mathfrak{s}_2$.
\end{proof}

The methods of this paper can show slightly more: we can get rid of the $\text{good}^+$.

\begin{thm}[Canonicity of good frames]\label{canon-good-frames}
  Let $\mathfrak{s}_1$, $\mathfrak{s}_2$ be weakly successful good $\lambda$-frames with underlying AEC $K$ and the same basic types. Then $\mathfrak{s}_1 = \mathfrak{s}_2$. 
\end{thm}
\begin{proof}[Proof sketch]
  Using Fact \ref{nf-frame}, let $\text{NF}_\ell$ be the non-forking frame extending $\mathfrak{s}_\ell$ for $\ell = 1,2$. Let $\1nf$, $\2nf$ be the independence relations (for models in $K_\lambda$) associated to $\text{NF}_1$, $\text{NF}_2$ respectively. By Fact \ref{nf-frame-respect2}, one can extend their domain to allow a single element on the left hand side. Thus without loss of generality we may assume $\mathfrak{s}_1$ and $\mathfrak{s}_2$ are type-full. Let $M \prec N \prec \bigN$ and let $a \in \bigN$. Assume $a \onf{\bigN}_M N_0$. We show $a \tnf{\bigN}_M N_0$. The symmetric proof will show the converse is true, and hence that $\mathfrak{s}_1 = \mathfrak{s}_2$. 

  First observe that stability, amalgamation, joint embedding and no maximal model in $\lambda$ implies we can build a saturated (hence model-homogeneous) model $\mathcal{M}$ of size $\lambda^+$. Since (as we will show) the argument below only uses objects of size $\lambda$, we can take $\mathcal{M}$ to be our monster model for this argument (i.e.\ we assume any set we consider comes from $\mathcal{M}$). Then we have $a \onf{\bigN}_M N_0$ if and only if $a \onf{\mathcal{M}}_M N_0$ so below we drop $\mathcal{M}$ and only talk about $a \1nf_M N_0$, and similarly for $\2nf$. Note that working inside $\mathcal{M}$ is not essential (we could always make the ambient model $\bigN$ grow bigger as our proof proceeds) but simplifies the notation and lets us quote our previous proofs verbatim.

  Now, we observe that our proof of Corollary \ref{endcor} is local-enough (i.e.\ it can be carried out inside $\mathcal{M}$). We sketch the details: First build a sequence $(M_i)_{i < \omega}$ independent (in the sense of $\2nf$) over $M$ so that $M_0 = N_0$, $M_i \equiv_M N_0$. Let $N := \bigcup_{i < \omega} N_i'$, where $N_i'$ witness the independence of the sequence. Notice that we can take $N \in K_\lambda$, by cardinality considerations. By extension, find $f: N \equiv_M N'$ so that $a \1nf_M N'$. Let $M_i' := f[M_i]$. By the proof of Lemma \ref{baldwin-218} (and recalling that $\kappa_1 (\nf) \le \omega$ in good frames), there is $i < \omega$ such that $a \2nf_M M_i'$ (notice that Fact \ref{nf-frame-respect2} is what makes the argument go through). Finally, use the proof of Lemma \ref{eplus-canon} to conclude that $a \2nf_M N_0$.
\end{proof}

We do not know whether one can say more, namely:

\begin{question}\label{categ-good-frame-q}
  Let $\mathfrak{s}_1$ and $\mathfrak{s}_2$ be good $\lambda$-frames with the same underlying AEC and the same basic types. Is $\mathfrak{s}_1 = \mathfrak{s}_2$?
\end{question}

\section{Conclusion}

We have shown that an AEC with a monster model can have at most one ``forking-like'' notion. On the other hand, we believe the question of when such a forking-like notion exists is still poorly understood. For example, is there a natural condition implying that coheir has extension in Fact \ref{bg-facts}? Even the following is open:

\begin{question}\label{indep-existence-q}
  Assume $K$ is fully $(<\kappa)$-tame, fully $(<\kappa)$-type short and categorical in some high-enough $\lambda > \kappa$. Does $K$ have an independence relation with $(E)$, $(U)$ and $(L)$?
\end{question}

Using the good frames machinery, an approximation is proven in \cite{ext-frame-jml} using some GCH-like hypotheses. However, the global assumptions of tameness and a monster model gives us a lot more power than just the local assumptions used to obtain a good frame. 

It is also open whether such an independence relation has to be coheir (i.e.\ even if coheir does not satisfy $(E)$):

\begin{question}\label{coheir-canon-q}
  Assume $\nf$ is an independence relation with $(E)$, $(U)$ and $(L)$. Let $M$ be sufficiently saturated. Under what conditions does $\nf_M = \ch_M$?
\end{question}

Finally, we note that while some of our results are local and can be adapted to the good frames context (see e.g.\ Theorem \ref{canon-good-frames}), some are not (e.g.\ Theorem \ref{symmetry}, Lemma \ref{inv-nfm}.\ref{trans-nfm}). It would be interesting to know how much non-locality is really necessary for such results. This would help us understand how much power the globalness of our definition of independence relations really gives us.

\bibliographystyle{amsalpha}
\bibliography{canon-forking}

\end{document}